\documentclass[11pt]{article}

\usepackage[top=50mm, bottom=50mm, left=50mm, right=50mm]{geometry}
\usepackage{lmodern}
\usepackage[utf8]{inputenc}
\usepackage[T1]{fontenc}
\usepackage{lineno}
\usepackage{amssymb}
\usepackage{amsmath}
\usepackage{amsthm}
\usepackage{epsfig}
\usepackage{graphicx}
\usepackage{graphics}
\usepackage{float}
\usepackage{subfigure}
\usepackage{multirow}
\usepackage{color}
\usepackage{fullpage}
\usepackage[normalem]{ulem} 
\usepackage{makeidx}
\usepackage{xspace}
\usepackage{wrapfig}
\usepackage{fouriernc}

\usepackage[cmtip,all]{xy}
\usepackage{mathtools}
\usepackage{graphicx}
\usepackage{tikz-cd}
\usepackage[new]{old-arrows}

\newtheorem{theorem}{Theorem}[section]

\theoremstyle{definition}

\newtheorem{conjecture}{Conjecture}[theorem]
\theoremstyle{remark}
\newtheorem{remark}[theorem]{Remark}

\numberwithin{equation}{section}

\makeindex

\definecolor{darkred}{rgb}{1, 0.1, 0.3}
\definecolor{darkblue}{rgb}{0.1, 0.1, 1}
\definecolor{darkgreen}{rgb}{0,0.6,0.5}

\newcommand {\mm}[1] {\ifmmode{#1}\else{\mbox{\(#1\)}}\fi}

\begin{document}

\title{Invariant Jet differentials and Asymptotic Serre duality}
 
\author{
{Mohammad Reza Rahmati,
  	 \footnote{email: mrahmati@cimat.mx (M. R. Rahmati)}}
}




\maketitle
\setcounter{page}{1}

\begin{abstract}
We generalize the main result of Demailly \cite{D2} for the bundles $E_{k,m}^{GG}(V^*)$ of jet differentials of order $k$ and weighted degree $m$ to the bundles $E_{k,m}(V^*)$ of the invariant jet differentials of order $k$ and weighted degree $m$. Namely, Theorem 0.5 from \cite{D2} and Theorem 9.3 from \cite{D1} provide a lower bound $\frac{c^k}{k}m^{n+kr-1}$ on the number of the linearly independent holomorphic global sections of $E_{k,m}^{GG} V^* \bigotimes \mathcal{O}(-m \delta A)$ for some ample divisor $A$. The group $G_k$ of local reparametrizations of $(\mathbb{C},0)$ acts on the $k$-jets by orbits of dimension $k$, so that there is an automatic lower bound $\frac{c^k}{k} m^{n+kr-1}$ on the number of the linearly independent holomorphic global sections of $E_{k,m}V^* \bigotimes \mathcal{O}(-m \delta A)$. We formulate and prove the existence of an asymptotic duality along the fibers of the Green-Griffiths jet bundles over projective manifolds.
We also prove a Serre duality for asymptotic sections of jet bundles. An application is also given for partial application to the Green-Griffiths conjecture.
\end{abstract}

\section{Introduction}
This text provides several results on the Green-Griffiths bundle of $k$-jets of entire holomorphic curves over a projective manifold $X$. First, we prove an equivariant (jet coordinate-free) version of Morse cohomology estimates of J. P. Demailly \cite{D1}, \cite{D2} for invariant k-jet metrics on Demailly-Semple bundles. The result will apply to the Green-Griffith conjecture on the entire curve locus in projective manifolds. We extend the main result of Demailly \cite{D2} for the bundles $E_{k,m}^{GG}(V^*)$ of jet differentials of order $k$ and weighted degree $m$ to the bundles $E_{k,m}(V^*)$ of the invariant jet differentials of order $k$ and weighted degree $m$. In this sense, Theorem 0.5 from \cite{D2} and Theorem 9.3 from \cite{D1} provide a lower bound $\frac{c^k}{k}m^{n+kr-1}$ on the number of the linearly independent holomorphic global sections of $E_{k,m}^{GG} V^* \bigotimes \mathcal{O}(-m \delta A)$ for some ample divisor $A$. We prove the same statement on a lower bound for linearly independent holomorphic global invariant sections on the same bundle. Notice that the group $G_k$ of local reparametrizations of $(\mathbb{C},0)$ acts on the $k$-jets by orbits of dimension $k$, so that there is an automatic lower bound $\frac{c^k}{k} m^{n+kr-1}$ on the number of the linearly independent holomorphic global sections of $E_{k,m}V^* \bigotimes \mathcal{O}(-m \delta A)$.

In the second result, we prove the existence of a Serre duality for the asymptotic sections of Green-Griffiths bundles over the projective variety $X$. As was mentioned, in \cite{D1}, \cite{D2}, the existence of asymptotic global sections for Green-Griffiths jets on a projective variety $X$ has been proved by using Morse estimates for the curvature of suitable metrics on these bundles. In other words $H^0(X,E_{k,m}^{GG}V^* \otimes A^{-1})=H^0(X_k,\mathcal{O}_{X_k}(m) \otimes \pi_{k}^*A^{-1}))$ are non-trivial when $m \gg k \gg 0$, where $A$ is a Hermitian ample line bundle on $X$. One wants to formulate a Serre duality for the sheaf of $k$-jet differentials. In order to state the Serre duality for asymptotic $k$-jets, one needs to check the existence of the dual asymptotic sections. The formulation of Serre duality is based on the existence of the canonical sheaf. One difficulty here is to define the relative canonical sheaf since the fibers of $X_k \to X$ have singularities. However, there is a generalization of the canonical sheaf defined for singular varieties, as explained in \cite{D1}, \cite{DR}. In this case, the relative canonical sheaf $K_{X_k/X}$ is replaced by the canonical sheaf $K_V$, where $V$ is a holomorphic subbundle of $T_X$, see \cite{DR}. We prove the existence of a Serre duality, written as
\begin{equation} \label{eq:serre-duality}
H^0(X_k,(\pi_k)_*\mathcal{O}_{X_k}(m)) \bigotimes H^{k(r-1)}(X_k, K_{X_k/X} \otimes (\pi_k)_*\mathcal{O}_{X_k}(-m')) \longrightarrow \mathcal{O}_X,
\end{equation} 
for $ m, m'\gg k \gg 0 $. In \cite{DR} Demailly and Reza Rahmati have proved the non-triviality of the cohomology groups in \eqref{eq:serre-duality}, see also \cite{D1}, \cite{D2}. We explain how this formula can be helpful toward a proof of the Green-Griffiths conjecture.

The third result is a formulation of a conjecture generalizing a theorem of J. Merker on the Green-Griffiths conjecture for generic projective hypersurfaces and a partial strategy to solve it. In \cite{M1}, J. Merker proves the Green-Griffiths-Lang conjecture for a generic hypersurface in $\mathbb{P}^{n+1}$. He demonstrated such a result for $X \subset \mathbb{P}^{n+1}(\mathbb{C})$, the universal family of hypersurfaces of degree $d$ as a generic member in the universal family $ \mathfrak{X} \subset \mathbb{P}^{n+1} \times \mathbb{P}^{(\frac{n+1+d)!}{(n+1)!d!}-1}$; by parametrizing all such hypersurfaces the GG-conjecture holds. In the proof by Merker for hypersurface case, the conjecture is established outside an algebraic subset $\Sigma \subset J_{\text{vert}}^n(\mathfrak{X})$ defined by vanishing of certain Wronskians, by using a result of Y. T. Siu in \cite{S}. Merker specifically proves that there are constants $c_n$ and $c_n'$ such that
$T_{J_{\text{vert}}^n(\mathfrak{X})} \otimes \mathcal{O}_{\mathfrak{X}_k}(c_n) \otimes \pi_{0k}^* L^{c_n'}$ is generated at every point by its global sections, where $L$ is an ample line bundle on $\mathfrak{X}$. One can ask if the same holds when $X \subset \mathbb{P}^{n+1}$ is a generic member of a family $\mathfrak{X}$ of projective varieties. We state this conjecture and sketch a methodology to approach this question using invariant metrics on the $k$-jet bundle. 

The fourth result is a finiteness theorem for the fiber ring of the moduli of $k$-jets as a differential ring. We prove the differential finiteness property of fiber rings in the moduli of jets of curves. The finite generation of fiber rings of $k$-jets is an open question due to the non-reductiveness of their group of symmetries concerning reparametrizations of the curves. The transformation group of $k$-jets is a non-reductive subgroup $G_k=\mathbb{C}^* \ltimes U_k$ of $GL_k(\mathbb{C})$, where $U_k$ is the unipotent radical consisting of upper-triangular $k \times k$ matrices of certain type. Unlike the reductive case, one can not deduce that the ring of polynomials invariant under the action of $G_k$, or its unipotent part, is finitely generated. The question is whether the ring of invariants $\mathbb{C}[f'(0),f''(0),...,f^{(k)}(0)]^{G_k}$ is finitely generated, where we have considered $f^{(k)}(0)$ as germ of variables. An attempt toward this conjecture has been made in \cite{BK}. The ring appears as the local ring of invariant sections of $J_k(X)=J_kT_X$ at a generic point $x \in X$. We prove that the above ring is differentially finitely generated.

The remainder of this paper is as follows. Section \ref{sec:Pre} consists of basic definitions and notations on the jet bundle of holomorphic curves. Section \ref{sec:results} consists of our main results. Specifically, we present four results in this section. Some conclusions are given in Section \ref{sec:conc}. Finally, in an appendix, we recall the basics of differential fields. 
\section{Preliminaries on Jet Bundles} \label{sec:Pre}
The Green-Griffith bundle associated to a pair $(X,V)$, where $X$ is a complex projective manifold (maybe singular), and $V$ is a holomorphic subbundle of $T_X$, the tangent bundle of $X$, with $rank(V)= r$, is a prolongation sequence of projective bundles 
\begin{equation}
\mathbb{P}^{r-1} \to X_k=P({V_{k-1}}) \stackrel{\pi_k}{\longrightarrow}  X_{k-1}, \qquad k \geq 1
\end{equation} 
inductively obtained making $X_k$ a weighted projective bundle over $X$ (see \cite{D1} and \cite{D2} for definitions). The bundles $\pi_k:X_k \to X$ provide a tool to study the locus of nonconstant holomorphic maps $f:\mathbb{C} \to X$ such that $f'(t) \in V$, since such a curve has a lift to $f_{[k]}:\mathbb{C} \to  X_k $ for every $ k $. Using the notation in \cite{D1}, we write $E_{k, m}^{GG}V^* =(\pi_{k})_* \mathcal{O}_{X_k}(m)$. Any entire curve $ f: \mathbb{C} \to X $ satisfies $Q(f_{[k]})=0$, where $f_{[k]}$ is a lift of $f$ and $Q$ is a global section of the jet bundle (see \cite{D1}). It is a conjecture due to Green-Griffiths since the total image of all these curves is included in a proper subvariety of $X$, provided $X$ is of \textbf{general type}. By general type, we mean $K_{V}$, the canonical bundle of $V$ is big, cf. \cite{D1}. One can define the Green-Griffiths bundle $X_k$ as follows,
\begin{equation}
X_k: = (J_kV \smallsetminus {0}) / \mathbb{C}^* ,
\end{equation}
where $J_k$ is the bundle of germs of $k$-jets of Taylor expansion for $f$. The projectivized bundle $J_k(X)/\mathbb{C}^*$ is called the Green-Griffiths bundle. As an alternative way to define these bundles is through their ring of sections. One can consider ring of weighted homogeneous polynomials $P(z,\xi)=\sum_a A(z) \xi^a$ in the variabales $\xi=(\xi_1,...,\xi_k)$ with weights $1,2,...,k$ and $a=(a_1,...,a_k)$ respectively, denoted by $E_k=\oplus_m E_{k,m}^{GG}$, where $m$ stands for the weighted degree. 

The following is the main theorem of \cite{D2} on the existence of asymptotic global sections of a twist of Green-Griffiths bundle $E_{k,m}^{GG}$.
\begin{theorem}\cite{D1, D2} \label{thm:Demailly}
Let $(X,V)$ be a directed projective variety such that $K_V$ is big, and let $A$ be an ample divisor. Then, for $k>>1$ and $\delta \in \mathbb{Q}_+$ small enough, and $\delta \leq c (\log k)/k$, the number of sections $h^0(X,E_{k,m}^{GG} \otimes \mathcal{O}(-m \delta A))$ has maximal growth, i.e. it is larger than $c_km^{n+kr-1}$ for some $m \geq m_k$, where $c,c_k >0, n=\dim (X), r= \text{rank}(V)$. In particular, the entire curves $f:\mathbb{C} \to X$ satisfy many algebraic differential equations.
\end{theorem}
The proof of the above theorem is mainly an estimate of the curvature of a suitable metric  ($k$-jet metric), namely $h_k$ on the Green-Griffiths jet bundles. The singularity locus for the metric $h_k$ which we denote by $\Sigma_{h_k}$ satisfies the inductive relation:
\begin{equation}
\Sigma_{h_k} \subset \pi_k^{-1}(\Sigma_{h_{k-1}}) \cup D_k,
\end{equation}
where $D_k=P(T_{X_{k-1}/X_{k-2}}) \subset X_k$. The divisors $D_k$ are the singularity locus of the projective jet bundle $X_k$ and their relation with the singularity of the k-jet metric is  $\mathcal{O}_{X_k}(1)=\pi_k^*\mathcal{O}_{X_{k-1}}(1) \otimes \mathcal{O}(D_k)$. 

The following theorem provides a method to read the entire curve locus from the singularities of suitable metrics on the jet bundle. See \cite{D1} and \cite{D2} for details.
\begin{theorem} \label{thm:curvelocus} (\cite{D1}, \cite{D2})
Let $(X,V)$ be a compact directed manifold. If $(X,V)$ has a $k$-jet metric $h_k$ with negative jet curvature, then every entire curve $f:\mathbb{C} \to X$ tangent to $V$ satisfies $f_k(\mathbb{C}) \subset \Sigma_{h_k}$, where $\Sigma_{h_k}$ is the singularity locus of $h_k$.
\end{theorem}

J. P. Demailly considers the jets of differentials that are also invariant under change of coordinate on $\mathbb{C}$. The invariant jet bundles are also refereed as \textit{Demailly-Semple (jet) bundles}. The bundle $J_k \to X$ of $k$-jets of germs of parametrized curves in $X$ has its fiber at $x \in X$, the set of equivalence classes of germs of holomorphic maps $f:(\mathbb{C},0) \to (X,x)$ with equivalence relation $f^{(j)}(0)=g^{(j)}(0), \ 0 \leq j \leq k$. By choosing local holomorphic coordinates around $x$, the elements of the fiber $J_{k,x}$ can be represented by the Taylor expansion:
\begin{equation}
f(t)=tf'(0)+\frac{t^2}{2!}f''(0)+...+\frac{t^k}{k!}f^{(k)}(0)+O(t^{k+1}) .
\end{equation}
Setting $f=(f_1,...,f_n)$ on open neighborhoods of $0 \in \mathbb{C}$, the fiber is 
\begin{equation} \label{eq:jet-vector}
J_{k,x}=\{(f'(0),...,f^{(k)}(0))\} = \mathbb{C}^{nk} .
\end{equation}

Let $G_k$ be the group of local reparametrizations of $(\mathbb{C},0)$
\begin{equation}
t \longmapsto \phi(t)=a_1t+a_2t^2+...+a_kt^k+..., \qquad a_1 \in \mathbb{C}^* .
\end{equation}
The following matrix multiplication gives its action on the $k$-jet vectors \eqref{eq:jet-vector},
\begin{equation}
[f'(0), f''(0)/2!, ..., f^{(k)}(0)/k!].
\left[ \begin{array}{ccccc}
a_1 & a_2 & a_3 & ... & a_k \\
0 & a_1^2 & 2a_1a_2 & ... & a_1a_{k-1}+...a_{k-1}a_1\\
0 & 0 & a_1^3 & ... & 3a_1^2a_{k-2}+...\\
. & . & . & ... & .\\
0 & 0 & 0 & ... & a_1^k 
\end{array} \right ].
\end{equation}
The group $G_k$ decomposes as $\mathbb{C}^* \times U_k$, where $U_k$ is the unipotent radical of $G_k$. 

Let $E_{k,m}$ be the Demailly-Semple bundle whose fiber at $x$ consists of $U_k$-invariant polynomials on the fiber coordinates of $J_k$ at $x$ of weighted degree $m$. Set $E_k=\bigoplus_m E_{k,m}$, to be the Demailly-Semple bundle of graded algebras of invariant jets. 
One way to produce invariant differentials is as follows. 
Assume  $P=P(f,f',...,f^{(k)})$ and $Q=Q(f,f',...,f^{(k)})$ are two local sections of the Green-Griffiths bundle, then the first invariant operator is $f \mapsto f_j'$. Define a bracket operation as follows 
\begin{equation} \label{eq:merker-bracket}
[P,Q]=\big(d \log \frac{P^{1/deg(p)}}{Q^{1/deg(Q)}} \big) \times PQ=\frac{1}{deg(P)}QdP-\frac{1}{deg(Q)}PdQ .
\end{equation}
This is compatible with Merker’s baracket $[P, Q]=\deg(Q)QdP- \deg(P)P dQ$ cf. \cite{M2}. If $(V,h)$ is a Hermitian vector bundle, the equations in \eqref{eq:merker-bracket} inductively define $G_k$-equivariant maps:
\begin{equation}
Q_k:J_kV \to S^{k-2} V \otimes \bigwedge^2 V, \qquad Q_k(f)=[f',Q_{k-1}(f)].
\end{equation}
The sections produced by $Q_k(f)$ generate the fiber rings of the Demailly-Semple bundle. Taking charts on the projective fibers, one can check that locally, the ring that these sections generate is equal to that of $J_k/G_k$, cf. \cite{M2}.
\section{Main Results} \label{sec:results}
This section contains our main results. All the materials in the subsequent subsections are new contributions to the literature. 

\subsection{Existence of global invariant asymptotic jets} 
There is a preference for a coordinate-free version of the bundle of $k$-jets of holomorphic curves on projective manifolds. This leads to considering the invariant sections of the Green-Griffiths bundles. Our first result is an invariant version of the Theorem \ref{thm:Demailly} of Demailly on Demailly-Semple bundles. This is presented in the following theorem.
\begin{theorem}[Main result]
The analogue of Theorem \ref{thm:Demailly} holds if the bundle $E_{k,m}^{GG}$ is replaced by $E_{k,m}$. Let's $(X,V)$ be a directed projective variety such that $K_V$ is big, and let $A$ be an ample divisor. Then for $k>>1$ and $\delta \in \mathbb{Q}_+$ small enough, and $\delta \leq c (\log k)/k$, the number of sections $h^0(X,E_{k,m} \otimes \mathcal{O}(-m \delta A))$ has maximal growth, i.e. it is larger than $c_km^{n+kr-1}$ for some $m \geq m_k$, where $c,c_k >0, n=\dim (X), r= \text{rank}(V)$. In particular, the entire curves $f:\mathbb{C} \to X$ satisfy many algebraic differential equations.
\end{theorem}
\begin{proof}
The proof consists in performing the calculations given in Theorem 9.3 of section 9 of \cite{D1} with an invariant metric. Toward this, we choose the metric in the form,
\begin{equation}
\left (\sum_{s=1}^k \epsilon_s \left (\ \sum_{\alpha} \mid P_{\alpha}(\xi) \mid ^2 \right )^{\frac{p}{w(P_{\alpha})}}\right )^{1/p},
\end{equation}
where $P_{\alpha}$ are a set of invariant polynomials in the jet coordinates. By Theorem \ref{thm:curvelocus} above, the Demailly-Semple locus of the lifts of entire curves is contained in $\Sigma_{h_k} \subset\  \{ P_{\alpha}=0 , \ \  \forall \alpha\}$, where the jet coordinates on the fibers $J_{k,x}$ are $\xi_1,...,\xi_k \in \mathbb{C}^n$, with $\xi_i=f^{(i)}(0)$ for an entire holomorphic curve $f : \mathbb{C} \longrightarrow X$ tangent to $V$. Besides, $\epsilon_1 \gg \epsilon_2 \gg . . . \gg \epsilon_k > 0$ are sufficiently small and $w(P_{\alpha})$ is the weight of $P_{\alpha}$. We make a specific choice of the $P(\alpha)$ as follows. Let us consider a change of coordinates:
\begin{equation} \label{eq:jetcoord}
(f_1,...,f_r) \longmapsto (f_1 \circ f_1^{-1},...,f_r \circ f_1^{-1})=(t,g_2,...,g_r)=\eta,
\end{equation} 
locally defined in a neighborhood of a point, where $n-r$ coordinates $f_i$ of $f(t) \in V_{f(t)}$ are completely determined by the remaining $r$ coordinates. The latter makes the first coordinate to be the identity. If we differentiate in the new coordinates, then all the resulting fractions are invariant of degree $0$,
\begin{equation}
g_2'=\frac{f_2'}{f_1'}\circ f_1^{-1}, \qquad g_2''=\frac{f_1'f_2''-f_2'f_1''}{f_1'^3},...
\end{equation} 
We take the $P_{\alpha}$'s to be all the polynomials that appear in the numerators of the components when we successively differentiate \eqref{eq:jetcoord} w.r.t $t$. An invariant metric in the original coordinates corresponds to a usual metric in the second one, subject to the condition that we need to make the average under the unitary change of coordinates in $V$. To calculate the curvature of the metric, we consider the frames of the vector bundle to be orthonormal. Recall that a change of coordinates on the manifold $X$ effects on jets as follows:
\begin{equation}
(\psi \circ f)^{(k)}(0)= \psi'(0). f^{(k)}+\text{higher order terms according to epsilons,}
\end{equation} 
with $\psi$ to be unitary. In our calculation, the effect of the change of variables in $X$ has only effect at the first derivative by composition with a linear map, up to the scaling epsilon factors (cf. \cite{D1}). Thus, we choose a metric in the form,
\begin{equation} \label{eq:metric}
\mid (z;\xi)\mid \sim \left (\sum_s  \epsilon_s \parallel \eta_s . (\eta_{11})^{2s-1}\parallel_h^{p/(2s-1)}\right )^{1/p} =\left (\sum_s  \epsilon_s \parallel \eta_s \parallel_h^{p/(2s-1)} \right )^{1/p}\mid \eta_{11} \mid,
\end{equation}
where $\eta_s$ are the jet coordinates $\eta_s = g^{(s)}(0), 1 \leq s \leq k$, induced by $g = (t, g_2(t), . . . , g_r(t))$. The weight of $\eta_s$ can be seen by differentiating \eqref{eq:jetcoord} to be equal to $(2s-1)$, inductively. Therefore, the above metric becomes similar to the metric used by Demailly in the new coordinates produced by $g$. We need to modify the metric in \eqref{eq:metric} slightly to be invariant under Hermitian transformations of the vector bundle $V$. The role of $\eta_{11}$ can be done by any other, $\eta_{1i}$ or even any other non-zero vector. To fix this issue, we consider,
\begin{equation}
\mid (z;\xi)\mid = \int_{\parallel v \parallel_1=1} \left (\sum_s  \epsilon_s \parallel \eta_s \parallel_h^{p/(2s-1)} \right )^{1/p}\mid <\eta_{1}.v> \mid^2,
\end{equation}
where the integration only affects the last factor making average over all vectors in $v \in V$. This will remove the former difficulty. 
The curvature is the same as for the metric in \cite{D1} but with only an extra contribution from the last factor, 
\begin{equation} \label{eq:invariant-metric}
\gamma_k(z,\eta)=\frac{i}{2\pi}\left ( w_{r,p}(\eta)+\sum_{lm\alpha}b_{lm\alpha}\left (\int_{\parallel v \parallel_1=1} v_{\alpha}\right )dz_l \wedge d\bar{z}_m +\sum_{s} \frac{1}{s}\frac{\mid \eta_s\mid^{2p/s}}{\sum_t \mid \eta_t \mid^{2p/t}} \sum c_{ij\alpha \beta}\frac{ \eta_{s\alpha}\bar{\eta}_{s\beta}}{\mid \eta_s \mid^2} dz_i \wedge d\bar{z}_j \right ).
\end{equation}
Namely, if $\pi_r : \mathbb{C}^{kr} \setminus {0} \longrightarrow \mathbb{P}(1^r, 2^r, . . . , k^r)$ is the canonical projection of the weighted projective space $\mathbb{P}(1^r, 2^r, ... , k^r)$ and
\begin{equation} 
\phi_{r,p}(z) := \frac{1}{p}\left (\log (\sum_{k}^{s=1}|z_s|^{\frac{2p}{s}} \right ),
\end{equation}
for some $p > 0$, then $w_{r,p}$ is the degenerate K\"ahler form on $\mathbb{P}(1^r, 2^r, . . . , k^r)$ with $\pi_r^*w_{r,p} = dd^c\phi_{r,p}$. We have $b_{l,m,\alpha} \in \mathbb{C}$. The contribution of the factor $\mid \eta_{11} \mid$ can be understood as the curvature of the sub-bundle of $V$, which is an orthogonal complement to the remainder. Thus, $b_{lm\alpha}=c_{lm\alpha \alpha}$, where $c_{lm11}$ is read from the coefficients of the curvature tensor of $(V,w^{FS})$, the Fubini-Study metric on $V$ (the second summand in \eqref{eq:invariant-metric}). Then, we need to look at the integral 
\begin{equation}
\int_{X_k,q}\Theta^{n+k(r-1)}=\frac{(n+k(r-1))!}{n!(k(r-1))!}\int_X \int_{P(1^r,...,k^r)}w_{a,r,p}^{k(r-1)}(\eta)1_{\gamma_k,q}(z,\eta) \gamma_k(z,\eta)^n.
\end{equation}

In the course of evaluating with the Morse inequalities, the curvature form is replaced by the trace of the above tensor in raising to the power $n=\dim X$. Then, if we use polar coordinates we have:
\begin{equation}
x_s=\parallel \eta_s \parallel^{2p/s}, \qquad u_s=\eta_s/\parallel \eta_s \parallel .
\end{equation}
Then, the curvature value when integrating over the sphere yields the following,
\begin{equation}
\gamma_k=\frac{i}{2\pi}\big(\sum_{lm}b_{lm\alpha} dz_l \wedge d\bar{z}_m+\sum_{s} \frac{1}{s}\sum c_{ij\lambda \lambda} u_{s\alpha}\bar{u}_{s \beta}dz_i \wedge d\bar{z}_j\big).
\end{equation}
Because the first term is a finite sum w.r.t. $1 \leq  \alpha \leq r$, and since $b_{lm\alpha}$ are labeled by $\alpha$, the estimates for this new form would be essentially the same as those in \cite{D1}. Therefore, one expects that
\begin{equation} \label{eq:Integralestimate}
\int_{X_k,q}\Theta^{n+k(r-1)}=\frac{(\log k)^n}{n!(k!)^r}\left ( \int_X 1_{\gamma,q}\gamma^n +O((\log k)^{-1}) \right),
\end{equation}
where $\gamma$ is the curvature form of $\text{det}(V^*/G_k)$ w.r.t the Chern connection of the determinant of the invariant metric, and $1_{\gamma,q}$ is the characteristic function of the set of those $(z, \eta)$, at which $\gamma$ is of signature $(n-q, q)$. The rest of the proof follows from the proof in \cite{D1} and \cite{D2}, and \eqref{eq:Integralestimate} above.
\end{proof}
\subsection{Asymptotic Serre duality for jet-bundles}\label{sec:duality}
Our second result is the theorem of Serre duality for asymptotic jet differentials. We wish to define the Serre duality for the fibers of the bundle of $k$-jets. The formulation of Serre duality is based on the existence of the canonical sheaf. A technical difficulty arises here in the definition of the relative canonical sheaf since the fibers of $X_k \to X$ have singularities. However, there is a generalization of the canonical sheaf definition for singular varieties, as explained in \cite{D1}, \cite{DR}. Thus, the relative canonical sheaf $K_{X_k/X}$ can be defined as the canonical sheaf $K_V$, where $V$ is a holomorphic subbundle of $T_X$ cf. \cite{DR, D1, D2}. Our result is as follows.
\begin{theorem} (Main result. Serre Duality on Jet Fibers)
There is a Serre duality for asymptotic cohomologies, $m, m' \gg k \gg 0$,
\begin{equation} \label{eq:Serreduality}
H^0(X_k,(\pi_k)_*\mathcal{O}_{X_k}(m)) \bigotimes H^{k(r-1)}(X_k, K_{X_k/X} \otimes (\pi_k)_*\mathcal{O}_{X_k}(-m')) \longrightarrow \mathcal{O}_X .
\end{equation} 
\end{theorem}
\begin{proof}
The Serre duality along fibers of jet bundles is based on the existence of the relative canonical sheaves $K_{X_k/X}$. Because the fibers in the Green-Griffiths bundle are weighted projective spaces of appropriate weights $(1,2,...,n)$, the relative Serre duality can be interpreted as the duality for coherent sheaves on weighted weights projective spaces. On account of this, we review the formulation of the adjoint pair along the fibers of $E_{k,m}^{GG}V^*$. According to the classical Serre duality, the dual pair associated to $H^{0} (F_x,(\pi_k)_* \mathcal{O}_{{X_k},x} (m))$ is $H^{k (r-1)} (F_x, K_{F_x} \otimes (\pi_k)_*\mathcal{O}_{{X_k},x} (- m))$, where $ K_{F_x} $ is the canonical sheaf of the fiber. In other words, we have
\begin{equation}
H^0(\pi_{k}^{-1}(x),(\pi_k)_* \mathcal{O}_{X_k}(m))^{\vee}=H^{k(r-1)}((\pi_{k}^{-1}(x), K_{F_x} \otimes (\pi_k)_*\mathcal{O}_{X_k}(-m)).
\end{equation}
By the Leray spectral sequence for $\pi_k:X_k \to X$ we get the following
\begin{equation}
H^{k (r-1)} (X_k, K_{X_k / X} \otimes (\pi_k)_*\mathcal{O}_{X_k} (- m ')) = H^0 .
\end{equation}
The duality in construction is the duality on each fiber of the jet bundle, all glued by the spectral sequence of Leray. The formula gives the adjoint pair on the fiber along the fiber $\pi_k^{-1}(x)$, 
\begin{equation}
H^0(\pi_{k}^{-1}(x),(\pi_k)_* \mathcal{O}_{X_k}(m))^{\vee}=H^{k(r-1)}((\pi_{k}^{-1}(x), K_{F_x} \otimes (\pi_k)_*\mathcal{O}_{X_k}(-m)).
\end{equation}
The degeneration of the Leray spectral sequence of the fibration $\pi_k:X_k \to X$ provides:
\begin{equation}
H^0(X_k,(\pi_k)_*\mathcal{O}_{X_k}(m))^{\vee} = H^0(X, R^{k(r-1)}(\pi_{k})_*( K_{X_k/X} \otimes (\mathcal{O}_{X_k}(-m))),
\end{equation}
which is equivalent to \eqref{eq:Serreduality}. We have to make sure that the sheaf
$H^{k(r-1)}(X_k, K_{X_k / X} \otimes (\pi_k)_*\mathcal{O}_{X_k} (- m')) $ is non-trivial, i.e., it has enough sections for $ m \gg k \gg 0 $. The nontriviality of the first factor in \eqref{eq:Serreduality} is proved in \cite{D1}. The nontriviality of the adjoint cohomology group in the pairing is obtained from the estimates:
\begin{equation}
H^q(X_k, K_{X_k/X} \otimes (\pi_k)_*\mathcal{O}_{X_k}(-m'))\geq \sum_{q-1,q,q+1}\frac{rm^n}{r!}\int_{X(\Theta,j)}(-1)^{q-j}\Theta^n-o(m^n),
\end{equation}
where $\Theta$ is the curvature of suitable $k-$jet metric on $X_k$, and $X(\Theta,j)=\{x \in X; \ \Theta \ \text{has signature} \ (n-j,j)\}$. We have considered the inequality for $q=k(r-1)$. It follows that when the $ K_{X_k / X} $ is big, both of the factors in the pairing \eqref{eq:Serreduality} are nontrivial for $m,m'\gg 0$, [cf. \cite{DR} section 4], and the proof is complete.
\end{proof}

In \cite{M1}, J. Merker shows that when $X$ is a hypersurface of degree $d$ in $ \mathbb{P}^{n + 1}$ and is a generic member of the universal family $\mathfrak{X} \subset \mathbb{P}^{n+1} \times \mathbb{P}^{N_d}$, the Green-Griffiths conjecture holds for $X$. His method uses ideas of Y. T. Siu on the existence of slanting vector fields, see \cite{S}, \cite{P}. Merker specifically proves that there are constants $c_n$ and $c_n'$ such that $T_{J_{\text{vert}}^n(\mathfrak{X})} \otimes \mathcal{O}_{\mathfrak{X}_k}(c_n) \otimes \pi_{0k}^* L^{c_n'}$ is generated at every point by its global sections, where $L$ is an ample line bundle on $\mathfrak{X}$. The proof of Merker \cite{M1} establishes the conjecture outside a certain algebraic subvariety $ \Sigma \subset J_{\text{vert}}^n (\mathfrak{X}) $ defined by Wronskians. An implication of Merker's result in \cite{M1} is the following,
\begin{equation}
H^0 \big (\mathfrak{X}_k,(\pi_k)_*\mathcal{O}_{\mathfrak{X}_k}(m) \big ) \bigotimes H^{k(r-1)}\big (\mathfrak{X}_k, K_{\mathfrak{X}_k/\mathfrak{X}} \otimes \big \langle J_{\text{vert}}^{k} (\mathfrak{X}) \big \rangle^{m'} \big ) \longrightarrow \mathcal{O}_\mathfrak{X} ,
\end{equation}
where $m' ,  m \gg k \gg 0$ and $\langle J_{\text{vert}}^{k} (\mathfrak{X}) \rangle^{m'}$ means the ring of operators generated by the $J_{\text{vert}}^{k} (\mathfrak{X})$ of degree $m'$. 
We note that any entire curve $f:\mathbb{C} \to X$ is satisfied by the sections  $ Q \in H^0(X_k, (\pi_k)_*\mathcal{O}_{X_k} (m)) $, i.e., $Q(f_{[k]})=0$ for $f_{[k]}:\mathbb{C} \to X_k$, a lift of $f$, cf. \cite{D1}. It follows that, in \eqref{eq:Serreduality} after the composition for $Q \in H^0 (X_k,(\pi_k)_*\mathcal{O}(X_k (m) )$ and  $P \in H^0 (X_k,K_{X_k/X} \otimes (\pi_k)_*\mathcal{O}(X_k (-m') )$, the resulting function also satisfies
\begin{equation}
\langle Q,P \rangle (f_{[k]})=0, \qquad Q=\sum_{|\alpha|=m}A_{\alpha}(z)\xi^{\alpha}, \ P= \sum_{|\beta|=m'}B_{\beta}(z)\partial_{\xi}^{\beta} .
\end{equation}
Therefore, the challenge is to show that the image of \eqref{eq:Serreduality} is a non-trivial ideal of $\mathcal{O}_X$. The above procedure plays a similar task as the idea on the existence of slanting vector fields along jet fibers, due to Siu \cite{S}. The slanting vector fields play the role of generating sections in the adjoint fiber rings of jet bundles.

\subsection{Generalization of the theorem of J. Merker for hypersurfaces}
As we mentioned in the section \ref{sec:duality} Merker \cite{M1} provided a direct proof for the Green-Griffiths conjecture on entire curve locus for a projective hypersurface of general type. In this section, we state a similar claim, as a conjecture, for a projective variety of general types, and we provide the motivations to prove it. To generalize the result of J. Merker on Green-Griffiths conjecture, we state the following more general conjecture on any generic projective variety.
\begin{conjecture} \label{con:gen-merker}
If $X \subset \mathbb{P}^{n+1}$ be a generic member of a family $\mathfrak{X}$ of projective varieties, then there are constants $c_n$ and $c_n'$ such that,
\begin{equation}
T_{J_{\text{vert}}^n(\mathfrak{X})} \otimes \mathcal{O}_{\mathfrak{X}_k}(c_n) \otimes \pi_{0k}^* L^{c_n'}
\end{equation}
is generated at every point by its global sections, where $L$ is an ample line bundle on $\mathfrak{X}$. We explain the way to understand and prove the conjecture \ref{con:gen-merker} as follows. By the analogy between microlocal differential operators and formal polynomials on the symmetric tensor algebra, it suffices to show that
\begin{equation} \label{eq:globaldual}
H^0(X_k, Sym^{\leq m'}\tilde{V}_k \otimes \mathcal{O}_{X_k}(m) \otimes \pi_{0k}^*B) \ne 0, \qquad m'>>m >>k,
\end{equation}
where $\tilde{V}_k$ is the in-homogenized $V_k$ as acting as differential operators in the first order. We also want to work over the Demailly-Semple bundle of invariant jets. 
\end{conjecture}
To prove the claim in \eqref{eq:globaldual}, notice the following. A similar procedure as the one of Demailly [\cite{D1} theorem 9.3] has to be applied to check the holomorphic Morse estimates for the following metric on the symmetric powers,
\begin{equation}
\vert (z,\xi) \vert= \Big( \sum_{s=1}^k \epsilon_s \big( \sum_{u_i \in S^sV^*} \vert W_{u_1,...,u_s}^s \vert^2 + \sum_{iju_{\alpha}u_{\beta}}C_{iju_{\alpha}u_{\beta}} z_i\bar{z}_j u_{\alpha}\bar{u}_{\beta}\big)^{p/s(s+1)} \Big)^{1/p},
\end{equation} 
where $W_{u_1,...,u_s}^s$ is the Wronskian 
\begin{equation}
W_{u_1,...,u_s}^s=W(u_1 \circ f,...,u_s \circ f),
\end{equation}
and we regard the summand front the $\epsilon_s$ as a metric on $S^sV^*$. We need to find estimates for the coefficients $C_{iju_{\alpha}u_{\beta}}$. Moreover, the frame $\langle u_i \rangle$ is chosen of monomials to be holomorphic and orthonormal at $0$ dual to the frame $\langle e^{\alpha}=\sqrt{l!/\alpha!}e_1^{\alpha_1}...e_r^{\alpha_r}\rangle$. The scaling of the basis in $S^lV^*$ is to make the frame to be orthonormal and is calculated as follows:
\begin{equation}
\langle e^{\alpha}, e^{\beta} \rangle=\langle \sqrt{l!/\alpha!}e_1^{\alpha_1}...e_r^{\alpha_r},\sqrt{l!/\beta!}e_1^{\beta_1}...e_r^{\beta_r} \rangle = \sqrt{1/\alpha! \beta!}\langle \prod_{i=1}^le_{\eta(i)}, \sum_{\sigma \in S_l} \prod_{i=1}^le_{\eta \circ \sigma(i)} \rangle ,
\end{equation}
via the embedding $S^lV^* \hookrightarrow V^{\otimes l}$ and the map $\eta:\{1....l\} \to \{1....r\}$ taking the value $i$ of the $\alpha_i$ times. 

Toward a Morse estimate with the coefficients of the curvature tensor, we proceed as follows. Because the frame $\langle e_{\lambda} \rangle$ of $V$ was chosen to be orthonormal at a given point $x \in X$, then, by substituting,
\begin{equation}
\langle e_{\lambda} , e_{\mu} \rangle=\delta_{\lambda \mu}+ \sum_{ij\lambda \mu}c_{ij\lambda \mu} z_i\bar{z}_j+...,
\end{equation} 
it follows that 
\begin{equation}
\langle e^{\alpha}, e^{\beta} \rangle= \sqrt{1/\alpha! \beta!} \left (\delta_{\alpha \beta} + \sum_{\eta \circ \sigma(i)=\eta(i)} c_{ij\alpha_{\eta(i)}\beta_{\eta(i)}}z_i\bar{z}_j+... \right ) .
\end{equation}
The strategy is to find the scalars $C_{iju_{\alpha}u_{\beta}}$ in terms of the curvature of the metric on $V$, in order to examine an estimation of the volume,
\begin{equation}
\int_{X_k} \Theta^{n+k(r-1)} = \frac{(n+k(r-1))!}{n!(k(r-1))!}\int_X \int_{\mathbb{P}(1^{[r]},...,k^{[r]})} \Theta_{\text{vert}}^{k(r-1)}\Theta_{hor}^n
\end{equation}
to be positive. However, the calculations with $\Theta$ involve more complicated estimates. Therefore, we pose the question of the existence of a positive lower bound for the global sections of the bundle in \eqref{eq:globaldual} as a step toward the Green-Griffiths conjecture. 
\begin{remark}
In \cite{DR} the existence of dual sections to $H^0(X,E_{k,m}^{GG}V^*)$ for $m \gg k \gg 0$ has been proved using the Morse inequalities, involving higher cohomologies of $X_k$ with similar coefficients, cf. loc. cit.
\end{remark}
\subsection{Differential finiteness property of fiber rings of jets moduli} \label{sec:diff-finite}
Our last result contributes to the finite generation of the fiber ring of regular jet curves. As mentioned before, the transformation group of $k$-jets is a non-reductive lie group. An open question is if the ring of invariants $\mathbb{C}[f'(0),f''(0),...,f^{(k)}(0)]^{G_k}$ is finitely generated, Where we have considered $f^{(k)}(0)$ as germ of variables. In the following, we prove this ring is differentially finitely generated; see the Appendix for the details and the basic definitions on differential fields.
\begin{theorem}
The local ring of invariant sections of $(E_{k, \leq m})_x^{G_k}$ is differentially finitely generated for each $m \in \mathbb{N}$. Furthermore, $\mathbb{C}[(J_{k,x}(X)]^{G_k}=\mathbb{C}\langle \wp_1,...,\wp_l \rangle (\alpha_1,...,\alpha_n)$, where $\wp_i$ s are polynomials on the Wronskians. The local ring of invariants $\mathbb{C}[(J_{k,x}(X)]^{U}$ under any subgroup of $U \subset SL_k$ is differentially finite generated.
\end{theorem} 
\begin{proof}
The fiber rings of the Green-Griffiths bundles $X_k$ and sheaves $E_{k,m}^{GG}(V^*)$ are differential rings. We shall consider their quotient fields. The algebraic groups $GL_k$, $SL_k$, and $G_k$ act linearly and are differential Galois groups. We explained that the fixed field of $SL_k$ and the Galois group $1$ are the fields generated by the fibers' Wronskians and the whole quotient field, respectively. Therefore, the middle group $G_k$ has a differentially finitely generated fixed field, where we have used the criteria given in Theorems \ref{thm:diff1}, and \ref{thm:diff2}. Furthermore, one finds that a possible choice of generators may include the fixed generators of $SL_k$, i.e., the Wronskians. By the Noether normalization theorem, there exists a finite number of generators $\wp_1,...,\wp_l$ such that the ring of fibers in $J_k(X)^{G_k}$ is algebraic over $\mathbb{C}\langle \wp_1,...,\wp_l\rangle$. It follows that,
\begin{equation}
\mathbb{C}[(J_{k,x}(X)]^{G_k}=\mathbb{C}\langle \wp_1,...,\wp_l \rangle (\alpha_1,...,\alpha_n)
\end{equation} 
are differential fields; the sections $\wp_i$ are local sections of the $J_k(X)$ and, as was explained, can be taken polynomials on Wronskians.
\end{proof}
\section{Conclusion}\label{sec:conc}
An invariant version of the main theorem of Demailly \cite{D1} on the existence of global sections of the twisted  $k$-jet bundle is proved using the Morse estimates on the curvature of an invariant metric on the Demailly-Semple bundle. Furthermore, we present a specific Serre duality for a bundle of $k$-jets on a projective variety. The proof is based on the existence of the relative pluri-canonical sheaf on jet fibers. We also provide an algebraic proof that the fiber ring on the regular locus of moduli of $k$-jets is differentially finitely generated.
\section{Appendix}
This appendix reviews basic definitions and facts about the differential rings and their differential Galois theory. The material is used in Section \ref{sec:diff-finite} of the text.

A differential field is a field $A$ with a derivation $\delta:A \to A, \ \delta(ab)=b\delta(a) + a \delta(b)$. Let $A$ be a differential field and $B$ a differential subfield. The differential Galois group, $G$ of $A/B$, is the group of all differential automorphisms of $A$ living $B$ fixed. Then, the same formalism as the Galois groups of usual fields appears also here. 

For any intermediate differential subfield $C$, we denote the subgroup of $G$ living $C$ elementwise fixed by $C'$; and similar for any subgroup $H$ of $G$ denoted by $H'$, the elements in $A$ fixed by that. Call a field or group closed if it equals its double prime. With these notations, PRIMED defines the Galois correspondence between closed subgroups and closed differential subfields [see \cite{K} for notations].

The Wronskian of $n$ elements $y_1,...,y_n$ in a differential ring is defined as the determinant,
\begin{equation}
W(y_1,...,y_n)=\left|
\begin{array}{cccc}
y_1 & y_2 & ... & y_n \\
y_1' & y_2' & ... &y_n' \\
   &     &     &    \\
y_1^{(n)} &     &   &y_n^{(n-1)} 
\end{array} 
\right|.   
\end{equation}

It is well known that $n$ elements in a differential field are linearly dependent over the field of constants if and only if their Wronskian vanishes \cite{K}. We recall that an extension of the form $A=K\langle u_1,...,u_n\rangle$, where $u_1,...,u_n$ are solutions of the differential equation 
\begin{equation}
L(y)=\frac{W(y,u_1,...,u_n)}{W(u_1,...,u_n)}=y^{(n)}+a_1y^{(n-1)}+...+a_ny=0
\end{equation}
is said a Picard extension, cf. \cite{K}.
\begin{theorem} \cite{K} \label{thm:diff1}
We have the following:
\begin{itemize}
\item[(1)] Let $K \subset L \subset M$ be differential fields. Suppose that $L$ is Picard over $K$ and $M$ has the same field of constants as $K$. Then, any differential automorphism of $M$ over $K$ sends $L$ into itself.
\item[(2)] The differential Galois group of a Picard extension is an algebraic matrix group over the field of constants. 
\item[(3)] If $K$ has an algebraically closed constant field of characteristic $0$, and $M$ is a Picard extension of $K$, then, any differential isomorphism over $K$ between two intermediate fields extends to the whole $M$. In particular, this also holds for any differential automorphism of an intermediate field over $K$. 
\item[(4)] Galois theory implements a one-to-one correspondence between the intermediate differential fields and the algebraic subgroups of the differential Galois group $G$. Thus, for example, a closed subgroup $H$ is normal iff the corresponding field $L/K$ is normal, then $G/H$ is the full differential Galois group of $L$ over $K$.
\end{itemize} 
\end{theorem}

In fact, over a constant field of $char=0$ any differential isomorphism between intermediate fields extends to the whole differential field.
Let $A=K\langle u_1,...,u_n\rangle$ be a Picard extension, and $W$ the Wronskian of $u_1,...,u_n$. A basic fact about the Wronskians is that for a differential automorphism $\sigma$ of $A$ we have $\sigma(W)=|c_{ij}|W$. Therefore, $W$ is fixed by $\sigma$ if and only if $ |c_{ij}| =1$. 

A family of elements $(x_i)_{i \in I}$ is called differential-algebraic independent if the family $(x_i^{(j)})_{i \in I,j \geq 0}$ is algebraically independent over the field of constants, otherwise we call them dependent. An element $x$ is called differentially algebraic if the family, consisting of $x$ only, is differential algebraic dependent. An extension is called differential algebraic if any element of it, is so. Finally, we say $G$ is differentially finite generated over $F$ if there exists elements $x_1,...,x_n \in G$ such that $G$ is generated over $F$ by the family $(x_i^{(j)})_{1 \leq i \leq n,j \geq 0}$, cf. \cite{K}.
\begin{theorem} \cite{K} \label{thm:diff2}
Let $F \subset G$ be an extension of differential fields, then
\begin{itemize}
\item If $G=F\langle x_1,...,x_n\rangle$ and each $x_i$ is differential algebraic over $F$, then, $G$ is finitely generated over $F$.
\item If $G$ is differential finite generated over $F$, and $F \subset E \subset G$ is an intermediate differential field, then, the $E$ is also differentially finite generated.
\end{itemize}
\end{theorem}


\begin{thebibliography}{}

\bibitem{S} Y. T. Siu, Hyperbolicity in complex geometry, The legacy of Niels Henrick Abel, Springer, Berlin, 2004, 543-566

\bibitem{D1}  J. P. Demailly, Hyperbolic algebraic varieties and holomorphic differential equations, Acta Math. Vietnam, 37(4) 441-512, 2012

\bibitem{D2}  J. P. Demailly, Holomorphic Morse inequalities and the Green-Griffiths-Lang conjecture , Prépublication Institut Fourier, November 15, 2010, math.AG:hal-00536319, arXiv: math.AG/1011.3636, dedicated to the memory of Eckart Viehweg; Pure and Applied Mathematics Quarterly 7 (2011) 1165-1208

\bibitem{DR} Jean-Pierre Demailly, Mohammad Reza Rahmati, Morse cohomology estimates for jet differential operators, Bollettino dell 'Unione Matematica Italiana, https://doi.org/10.1007/s40574-018-0172-2

\bibitem{GG} M. Green, P. Griffiths, Two applications of algebraic geometry to entire holomorphic mappings, The
Chern Symposium 1979. (Proc. Intern. Sympos., Berkeley, California, 1979) 41-74, Springer, New York, 1980.

\bibitem{M1} J. Merker, low pole order frameson vertical jets of the universal hypersurface, Ann. Institut Fourier (Grenoble), 59, 2009, 861-933 

\bibitem{M2} J. Merker, Application of computational invariant theory to Kobayashi hyperbolicity and to Green–Griffiths algebraic degeneracy, Journal of Symbolic Computation 45 (2010) 986–1074 

\bibitem{P} M. Paun, Vector fields on the total space of hypersurfaces in the projective space and hyperbolicity, Math. Ann. (2008) 340:875-892

\bibitem{RT} M. Rossi, L. Terracini, Weighted projective spaces from toric point of view and computational aspects, preprint

\bibitem{BK} G. Berczi, F. Kirwan, A geometric construction for invariant jet differentials, arXiv:1012.1797v3, preprint    

\bibitem{K} I Kaplansky, An introduction to differential algebra, Pub. de L'institut de Math. de Univ. Nancao, Hermann Paris 1957


\end{thebibliography}

\bibliographystyle{amsplain}

\end{document}